\documentclass[11pt]{article}
\usepackage[utf8]{inputenc}	
\usepackage{amsmath,amsthm,amsfonts,amssymb,amscd}
\usepackage{multirow,booktabs}
\usepackage[table]{xcolor}
\usepackage{esint}
\usepackage{fullpage}
\usepackage{lastpage}
\usepackage{enumitem}
\usepackage{fancyhdr}
\usepackage{mathrsfs}
 \usepackage{outline}

\usepackage{wrapfig}
\usepackage{setspace}
\usepackage{calc}
\usepackage[english]{babel}
\newtheorem{theorem}{Theorem}
\newtheorem{corollary}{Corollary}[theorem]
\newtheorem{lemma}[theorem]{Lemma}
\theoremstyle{definition}

\usepackage{cancel}
\usepackage[retainorgcmds]{IEEEtrantools}
\usepackage[margin=3cm]{geometry}
\usepackage{mathtools}
\usepackage{amsmath}

\usepackage{empheq}
\usepackage{framed}
\usepackage[most]{tcolorbox}
\usepackage{xcolor}
\colorlet{shadecolor}{orange!15}
\theoremstyle{definition}

\usepackage{textalpha}
\usepackage[hyphens]{url}
\usepackage{hyperref}
\usepackage[hyphenbreaks]{breakurl}

\usepackage[
backend=biber,
style=numeric,
sorting=ynt
]{biblatex}

\bibliography{ref.bib}

\usepackage{setspace}

\begin{document}

\thispagestyle{empty}

\begin{center}
{\LARGE \bf Polynomial formulations of Multivariable Arithmetic Progressions of Alternating Powers}\\
Brian Nguyen
\end{center}

\begin{abstract}
    Taking inspiration from the work of Lanphier \cite{LANPHIER2022125716}, a generalized multivariable polynomial formulation for sums of alternating powers is given, as well as analogous sums. Furthermore, an analog of the Euler-Maclaurin Summation Formula is established and used to give asymptotic formulas for multivariable-type Lerch-Hurwitz zeta functions.
\end{abstract}
\section{Introduction}
\noindent \indent In this paper, we give a multivariable generalization of a class of analogs to Faulhaber's formula, most notably the alternating case. More precisely, let $A_r = (a_1, \cdots, a_r), N_r = (n_1, \cdots n_r) \in \mathbb{Z}_{\geq 0}^r$ be vectors, $x \in \mathbb{R}$, $s \in \mathbb{Z}_{\geq 0}$, and an imaginary constant $c$. Let $b_r = \underbrace{(b, b, \cdots ,b)}_{r b\text{'s}}$, and for a pair of $r$-dimensional vectors let $\cdot$ be the standard dot product. We give a polynomial formulation on $s$ for the sum
\begin{equation*}
    \sum_{M_r = 0_r}^{N_r} (A_r \cdot M_r + x)^s e^{c A_r \cdot M_r},
\end{equation*}
where $\sum_{M_r = 0_r}^{N_r} = \sum_{m_1 = 0}^{n_1}  \cdots \sum_{m_r = 0}^{n_r}$ denotes the multi-sum. Additionally, for a $q$-times differentiable function $g$ defined over an interval $[lm,ln]$, where $m,n,l$ are integers and $lm < ln$, and a rational $\theta \in \mathbb{Q} \backslash \mathbb{Z}$, we give a formula for $\sum_{r = lm+1}^{ln} g(r)e^{2i \pi \theta r} $ in a manner analogous to the Euler-Maclaurin Summation formula. Using this result, we may loosen restrictions on $s$ and find asymptotic polynomial formulations of the sum above for $s \in \mathbb{C} \backslash \mathbb{Z}$ and $Re(s) > - 1$, with the additional condition that $c$ is a rational multiple of $2 i \pi$ and the product of each entry of $A_r$ with $c$ is not an integral multiple of $2 i \pi$.

The Bernoulli Numbers, denoted as $B_k$, are a sequence of numbers that were discovered in an attempt to find a general formula for the sum $\sum_{m=1}^n m^k$, where $n$ and $k$ are positive integers. While these values are refer to Jakob Bernoulli, they were actually initially discovered by mathematician Seki Takakazu (1642 - 1708) who happened to also be working on the sum of powers as well \cite{7b2ec4b0-bc03-30e6-ab16-1792c3f36b3f}. In 1755, Euler defined the Bernoulli numbers by the generating function
\begin{equation}\label{bernNum}
    \frac{z}{e^z-1} = \sum_{k=0}^{\infty} B_k \frac{ z^k}{k!},
\end{equation}
for $z \in \mathbb{C},$ where the $B_k$'s can be computed via Taylor Series expansion. Accompanied by the Bernoulli numbers are the Bernoulli polynomials, $B_k(y)$ \cite{arakawa2014bernoulli}, defined by the function \cite{Zagier2014CuriousAE}
\begin{equation*}
    \frac{ze^{zy}}{e^z-1} = \sum_{k=0}^{\infty}  B_k(y) \frac{z^k}{k!}.
\end{equation*}
It is well known that $B_m(y) = \sum_{k=0}^m \binom{m}{k}B_{m-k}x^k$. It was Johann Faulhaber who initially made significant strides on the problem of the sum of integer powers, and in 1631 masterwork \textit{Academia Algrebrae} \cite{faulhaber1631academia}, he established formulas for powers up to 23. Taking inspiration from his work, Bernoulli found a closed expression for the sum $\sum_{m=1}^n m^k$ for all $k$, now known as Faulhaber's formula \cite{Knuth_1993}, which is commonly expressed as
\begin{equation*}
    \sum_{m=1}^n m^k = \frac{1}{k+1}(B_{m+1}(k+1)-B_{m+1}(1)).
\end{equation*}

Variations on Bernoulli numbers and Faulhaber's formula were developed as well (see \cite{Jang_2012}, \cite{Petersen_2015}, \cite{Kim_Kwon_Lee_Seo_2014}). Dominic Lanphier \cite{LANPHIER2022125716} had used the properties of  the Bernoulli-Barnes numbers to find a formula for multivariable generalization of Faulhaber's. Alternating sums of powers of integers such as $\sum_{k =0}^n (-1)^m m^k  $; have also been studied, as seen in \cite{kim2005note}; for this problem, Euler numbers $E_m$ and Euler Polynomials $E_m(x)$ were defined with the generating functions
    \begin{equation}\label{1eq}
        \frac{2}{e^z+1} = \sum_{n =0}^{\infty} E_n \frac{z^n}{n!},
    \end{equation}
and
    \begin{equation}\label{2eq}
        \frac{2e^{xz}}{e^z+1} = \sum_{n=0}^{\infty} E_n(x)\frac{z^n}{n!}.
    \end{equation}
\indent  One can see the numerous properties exhibited by Euler numbers and polynomials (see \cite{Srivastava_Choi_2012}), as well as its many connections with Bernoulli polynomials (see \cite{Brillhart+1969+45+64}). In particular, Euler polynomials can be used to give a closed expression for the case of the alternating sum \cite{kim2005note} in a manner analogous to Faulhaber's formula. A common expression for this formula is
\begin{equation*}
  \sum_{k =0}^n (-1)^m m^k  = \frac{1}{2}(E_p(n) + (-1)^n E_m(n+1) ).
\end{equation*}

Various other literature such as \cite{Liu}, \cite{Barbero_G__2020}, appears to cite a different definition for Euler numbers instead, as given below:
\begin{equation*}
    \frac{2}{e^z + e^{-z}} = \sum_{n=0}^{\infty} E_n^* \frac{z^n}{n!}.
\end{equation*}
To ensure distinction, we will denote this set of Euler numbers as $E_n^*$ instead; our primary focus will be on $E_n$.

Taking inspiration from the mathematical structure of (\ref{1eq}) and (\ref{2eq}) as well as  Dominic's work \cite{LANPHIER2022125716}, we give a generalized definition of Euler numbers. Let $A_r = (a_1,a_2, \cdots, a_r) \in \mathbb{Z}_{\geq 0}^r$. We define the generalized Euler numbers by the following generating function:
\begin{equation*}
    \frac{2^r}{\prod_{l=1}^r (1-e^{a_l(z+j)})} = \sum_{m=0}^{\infty} E_m(j,A_r)\frac{z^m}{m!},
\end{equation*}
and similarly for polynomials
\begin{equation*}
    \frac{2^re^{xz}}{\prod_{l=1}^r (1-e^{a_l(z+c)})} = \sum_{m=0}^{\infty} E_m(x,j; A_r)\frac{z^m}{m!},
\end{equation*}
where $c$ is an imaginary constant. Taking $r = 1$, $j = i\pi$, and $A = \langle 1 \rangle$, we get (\ref{1eq}) and (\ref{2eq}). Note that the generalized Euler numbers can be expressed as
\begin{equation*}
    E_m(j, A_r) = \sum_{\substack{l_1 + \cdots + l_r = m \\ 0 \leq l_p \leq m}} \binom{m}{l_1, \cdots ,l_r} E_{l_1}(j,a_1) \cdots E_{l_r}(j, a_r),
\end{equation*}

where the multinomial coefficient is defined as $\binom{m}{l_1, l_2, \cdots , l_k} = \frac{m!}{l_1!l_2! \cdots l_k!}$. In a similar fashion as \cite{LANPHIER2022125716}, the generating function from equation \ref{2eq} implies that for any partition $x = x_1 + \cdots + x_d$ with $0\leq d \leq r$ and $A_1', A_2', \cdots , A_d'$ are a set of vectors that form a partition of $A_r$, we have 

\begin{equation*}
    E_m(x,j;A_r) = \sum_{\substack{l_1 + \cdots + l_d = m \\ 0 \leq l_p \leq m}} \binom{m}{l_1, \cdots , l_l} E_{l_1}(x_1,j;A_1') \cdots E_{l_d} (x_l,j ; A_l').
\end{equation*}

 For a set $S$, let $|S|$ be the number of elements in the set. We may have the following result:

\begin{theorem}
Let $N_r = (n_1,n_2, \cdots n_r) \in \mathbb{Z}_{\geq 0}, A_r = (a_1, a_2, \cdots a_r)  \in \mathbb{Z}_{> 0}^r$, $x \in \mathbb{Z}_{\geq 0}$, $s \in \mathbb{Z}_{>0}$, and $c$ to be a purely imaginary number that is not a integral multiple of $2\pi i$. For a set $S = \{i_1, \cdots i_{|S|} \} \subseteq \{ 1,2, \cdots, r\}$, define $A_s$ to be $(a_{i_1}, a_{i_2} \cdots, a_{i_{|S|})}$; define $N_S$ analogously as $(n_{i_1}, n_{i_2} \cdots, n_{i_{|S|})}$. Let $P_S = A_S \cdot N_S$. Then,
    \begin{equation*}
        \sum_{M_r = 0_r}^{N_r} (A_r 
\cdot M_r+x)^s e^{c(A_r 
\cdot M_r)} = \frac{1}{2^r} \sum_{S \subseteq \{1,2, \cdots r\}} (-1)^{|S|} e^{c( P_S+ A_S \cdot1_{|S|} -x)} E_{s}(P_S+ A_S \cdot 1_{|S|}+x,j;A_r).
    \end{equation*}
    In the case when $c = i\pi$ we attain the generalization of alternating sums of powers:
        \begin{equation*}
        \sum_{M_r = 0_r}^{N_r} (A_r 
\cdot M_r+x)^s(-1)^{A_r 
\cdot M_r} = \frac{1}{2^r} \sum_{S \subseteq \{1,2, \cdots r\}} (-1)^{|S|+ A_S \cdot (N_s+1_{|S|} -x)}  E_{s}(A_S \cdot (N_S+1_{|S|})+x,i\pi;A_r),
    \end{equation*}
where
\begin{equation*}
        \frac{2^re^{xz}}{\prod_{j=1}^r (1-e^{a_rz} (-1)^{a_r} )} = \sum_{m=0}^{\infty} E_m(x,i\pi ; A_r)\frac{z^m}{m!}.
    \end{equation*}
\end{theorem}

\begin{corollary}
    Setting $N_r = 0_r,$ we have that
    \begin{equation*}
     \sum_{S \subseteq \{1,2, \cdots r\}} (-1)^{|S|} e^{j( A_S \cdot1_{|S|})} E_{m}( A_S \cdot 1_{|S|}+x,j;A_r) =2^r x^m.
    \end{equation*}
\end{corollary}

One limitation for Theorem 1 is that the exponent $m$ in $(A_r \cdot M_r+x)^m$ is an integer. There exists various formulations/approximations to compute arithmetic progressions of non-integral powers; for example, Parks \cite{PARKS2004343} demonstrates asymptotic formulas to compute the sums of non integral powers of the first $n$ positive integers. To rectify for this issue, we explore another class of values.

Let $k$ be a positive integer, $a$ be a nonzero integer that is not a multiple of $k$, and $x\in \mathbb{C}$. Define the polynomials $C_{n,k}(x;a)$ and $\tilde{C}_{n,k}(x;a)$ as
    \begin{equation*}
        C_{n,k}(x;a) = \sum_{l=0}^{k-1} B_n(x- \frac{l}{k})  e^{2i\pi a l/k}
    \end{equation*}
    
    \begin{equation*}
        \tilde{C}_{n,k}(x;a) = \sum_{l=0}^{k-1} \tilde{B}_n(x- \frac{l}{k}) e^{2i\pi a l/k},
    \end{equation*}
    where $B_k(x)$ is the $k$th Bernoulli polynomial and $\tilde{B}_k(x) = B_k(\{x\})$ is the periodic Bernoulli function. Furthermore, let $C_{n,k}(a) = C_{n,k}(0,a)$. Applying equation \ref{bernNum}, one can see that $C_{n.k}(x;a)$ admits the generating function
    \begin{equation}\label{genfun}
            \sum_{n=0}^{\infty} C_{n,k}(x;a) \frac{z^n}{n!}
            = \left( \frac{ze^{xz}}{e^z-1} \right) \left( \frac{e^{-z}-1}{e^{2ia\pi/k}e^{-z/k}-1} \right),
    \end{equation}
   Through these values, we may achieve a formula similar to the Euler-Maclaurin Summation formula:
\begin{theorem}\label{theorem2}
    Let $f: \mathbb{R} \rightarrow \mathbb{C}$ be a $q-1$ differentiable function over $[m,n]$, where $m,n \in \mathbb{Z}$ and $m < n$. Then
   \begin{equation}
          \sum_{r = m}^n  \sum_{l=1}^{k} e^{2i\pi al/k} f(r + l/k) = \sum_{l=1}^q  \frac{C_{l,k}(a)(-1)^l}{l!} \left[f^{(l-1)}(n) - f^{(l-1)}(m) \right] + R 
\end{equation}
where $R = \frac{(-1)^{q+1}}{q!} \int_m^n \tilde{C}_{q,k}(x) f^{(q)}(x) dx$. As an alternative formulation, let $g: \mathbb{R} \rightarrow \mathbb{C}$ be a $q-1$ differentiable function over $[mk, nk]$. Then
\begin{equation}
    \sum_{r = mk+1}^{nk} e^{2i\pi a r/k} g(r) = \sum_{l=1}^q \frac{C_{l,k}(a)(-1)^lk^{l-1}}{l!} \left[g^{(l-1)}(kn) - f^{(l-1)}(km) \right] + R,
\end{equation}
where $R = \frac{(-1)^{q+1}k^q}{q!}. \int_{km}^{kn} \tilde{C}_{q,k}(x) f^{(q)}(x) dx.$
\end{theorem}
Making further use of $C_{n,k}(x;a),$ we define $C_{m,k}^*(a) = C_{m,k}(a)a^{m-1}$ and extend the input to $r$-dimensions to $A_r$ by
    \begin{equation*}
        C_{m,k}^*(A_r) = \sum_{\substack{l_1 + \cdots + l_r = m \\ 0 \leq l_p \leq m}} \binom{m}{l_1, \cdots, l_r}  C_{l_1,k}^*(a_1) \cdots C_{l_r,k}^*(a_r).
    \end{equation*}
From equation (\ref{genfun}), one may derive that
\begin{equation*}
    \sum_{n=0}^{\infty} C_{n,k}^*(A_r) \frac{z^n}{n!} = \prod_{p=1}^r \frac{z}{e^{a_pz}-1} * \frac{e^{-a_p z}-1}{e^{2ia_p \pi/k}*e^{-a_p z/k} -1}.
\end{equation*}
 Furthermore, let $C_{s,m,k}^*(x;A_r)$ as
    \begin{equation*}
        C_{s,m,k}^*(x;A_r) = \sum_{j=0}^m (-k)^j\binom{s}{j}C_{j,k}^*(A_r)x^{s-j}.
    \end{equation*}

For $s \in \mathbb{C},$ $x \in \mathbb{R}_{\geq 0}$ the Euler-Zeta function, also known as the Dirichlet ETA function, is defined as
    \begin{equation}
\zeta_E(s,x) = 2\sum_{n=0}^{\infty} \frac{(-1)^n}{(n+x)^s}.
\end{equation}
Euler numbers and Euler polynomials share an intimate relationship with the zeta function, as seen in \cite{kim2005note}. Again, we generalize the Euler-Zeta function in the following manner: let $A_r \in \mathbb{Z}_{>0 0}^r$ for a positive integer $r$, and let $c \in \mathbb{R}$ such that $c \geq 0$. We define the generalized Euler-Zeta function as
\begin{equation*}
    \zeta_E(s,x,c;A_r) := 2^r \sum_{M_r \in \mathbb{Z}_{\geq 0}^r} \frac{e^{c(M_r \cdot A_r})}{(M_r\cdot A_r+x)^s},
\end{equation*}
where $c$ is an imaginary constant. Now, let $(a)_r = \Gamma(a+r)/\Gamma(a)$, where $(a)_0 = 1$; this is commonly known as the Pochhammer symbol. Given this machinery, we have the following result:

\begin{theorem}
    Let $Re(s) > -1$, $k \in \mathbb{Z}_{> 1}$, $A_r = (a_1, \cdots, a_r) \in \mathbb{Z}_{> 0}^r$ where $k \nmid a_i$, $x \in \mathbb{R}_{\geq 0}$ and $q = \lceil Re(s) \rceil $. Define $P_S = A_S \cdot N_S$. We have
\begin{equation*}
    \begin{split}
        \sum_{M_r = 0_r}^{N_r} (A_r \cdot M_r+x)^s * e^{\frac{2\pi i}{k}(A_r \cdot M_r)} &= \frac{1}{k^{r-1} (s+2)_{r-1}} \sum_{\substack{S \subseteq \{1,2, \cdots r\} \\ S \neq \emptyset}} (-1)^{|S|}e^{\frac{2\pi i}{k}(P_S+ A_S \cdot 1_{|S|}))} C_{s+r,q,k}^* (P_S+x;A_r) \\
        & + \frac{1}{2^r} \zeta_E(-s, x; A_r) + O(N_{min}^{Re(s)-q+2}).
    \end{split}
\end{equation*}
\end{theorem}

\begin{theorem}
For $Re(s) > -1$, $k \in \mathbb{Z}_{> 1}$, $A_r = (a_1, \cdots, a_r) \in \mathbb{Z}_{> 0}^r$ where $k \nmid a_i$, $x \in \mathbb{R}_{\geq 0}$ and $q = \lceil Re(s) \rceil $, we have
    \begin{equation*}
     \zeta_E(-s,x,\frac{2 \pi i}{k};A_r) =  \frac{2^r}{k^{r-1}(s+2)_{r-1}}  C_{s+r,q,k}^* (x - A_r \cdot 1_r; A_r) + O(x^{Re(s)-q+2}).
\end{equation*}
\end{theorem}

These results can be useful for computations. As an example, suppose $r = 2$, $A_r = (1,3)$, $N_r = (100,150)$, and $m =2$. Applying Theorem 1, we may find that
    \begin{equation*}
        \begin{split}
              \sum_{(n_1,n_2) = (0,0)}^{(100,150)} (3n_1+n_2)^2 *(-1)^{3m+n} &= \frac{1}{4} (E_2(0,i\pi, A_r)E_2(151,i\pi, A_r)) \\
              & + E_2(303,i\pi, A_r) + E_2(484,i\pi, A_r)) \\
                &= 79275,
        \end{split}
    \end{equation*}
where the last line is computed by $   E_2(x,i\pi; A_r) =  6 - 16x +4x^2$ as generated by Wolfram Mathematica.

\section{Proof of Theorem 1}
\noindent \indent To prove theorem 1, we must first establish some properties concerning the class of multivariable zeta functions and Euler polynomials. Let $P$ be the path of integration that trails from positive infinity down the positive axis, hovering above by $\epsilon > 0$, circles counterclockwise around the origin by a radius of $\delta$ until it is below the positive real axis by $\epsilon$, and trails back to infinity.We extend $\zeta_E$ to evaluate for negative $s$:

\begin{lemma}\label{lemma1}
     Let $c \in \mathbb{R}_{\geq 0}$, $j \in i\mathbb{R}$, and $A_r = (a_1, \cdots , a_r) \in \mathbb{C}^r$. Then $\zeta_E$ may be extended for $s < 0$ by the integral formulation
    \begin{equation*}
   \zeta_E(s,c,j;A_r) = \frac{\Gamma(1-s)}{2\pi i} \oint_P e^{-cz}(-z)^{s} \prod_{i=1}^r \frac{2}{1-e^{-a_i(z+j)}} \frac{dz}{z},
\end{equation*}
where $\Gamma(s) = \int_{0}^{\infty} e^{-z}z^{s-1} dz$ is the gamma function. Note that the choice of $\delta$ and $\epsilon$ may be arbitrary as the choice of $\delta$ and $\epsilon$ is arbitrary as changing these parameters preserves homotopy on the domain of $\zeta_E$.

\end{lemma}
\begin{proof}
  
We follow a similar argument as section 1.4 in \cite{edwards_1982}. Consider $\Gamma(s) = \int_{0}^{\infty} e^{-z}z^{s-1} dz$; note that for a value $m$ such that $Re(m) > 0$, mapping $z \rightarrow mx$ on the right hand side yields
\begin{equation*}
\int_{0}^{\infty} e^{-z}x^{s-1} dz =
    \int_0^{\infty} e^{-mz}(mz)^{s-1}*m dz = m^s \int_0^{\infty} e^{-mz}z^{s-1}dz,
\end{equation*}

\begin{equation*}
   \Rightarrow \int_0^{\infty} e^{-mz}z^{s-1}dz = \frac{\Gamma(s)}{m^s},
\end{equation*}
\begin{equation*}
    \int_0^{\infty}  e^{mj-mz}z^{s-1} dz= \frac{\Gamma(s)*e^{mj}}{m^s}.
\end{equation*}
Now, let $A_r = (a_1,a_2, \cdots, a_r) \in \mathbb{Z}_{> 0}^r$, where every entry has a positive real part, and let $M_r = (m_1,m_2,\cdots, m_r) \in \mathbb{Z}_{\geq 0}^r$. Let $\cdot$ be the dot product; then for some $c \geq 0$,
\begin{equation*}
    \int_0^{\infty}  e^{(j-z)(A_r \cdot M_r + c)}z^{s-1}dz = \frac{\Gamma(s) e^{j(M_r\cdot A_r + c})}{(A_r \cdot M_r + c)^s}.
\end{equation*}
Summing over every $M_r \in \mathbb{Z}_{\geq 0}^r$, we have
\begin{equation*}
    \begin{split}
        \Gamma(s) e^{jc} \sum_{M_r \in \mathbb{Z}_{\geq 0}^r} \frac{e^{j(M_r \cdot A_r })}{(M_r\cdot A_r+c)^s} &= 
    \int_0^{\infty} \sum_{M_r \in \mathbb{Z}_{\geq 0}^r} e^{j(A_r \cdot M_r)} e^{-z(A_r \cdot M_r + c)}z^{s-1}dx, \\ 
    &= \int_0^{\infty} e^{-cx}x^{s-1} \prod_{i=1}^r \sum_{m_i \geq 0} (e^{a_i(x+j)})^{m_i} dx, \\
    &= \int_0^{\infty} e^{-cx}x^{s-1} \prod_{i=1}^r \frac{1}{1-e^{-a_i(x+j)}} dx.
    \end{split}
\end{equation*}
Thus,
\begin{equation*}
     \zeta_E(s,c;A_r) = \frac{1}{e^{jc}\Gamma(s)} \int_0^{\infty} e^{-cx}x^{s-1} \prod_{i=1}^r \frac{2}{1-e^{-a_i(x+j)}} dx.
\end{equation*}
Now, consider the contour integral
\begin{equation*}
    \oint_P e^{-cx}(-x)^{s} \prod_{i=1}^r \frac{2}{1-e^{-a_i(x+j)}} \frac{dx}{x}.
\end{equation*}
We examine when $\epsilon \rightarrow 0$; the integral can then be partitioned as $\oint_P f = \oint_{+ \infty}^{\delta} f + \oint_{|x| = \delta}f + \oint_{\delta}^{\infty} f,$ where $f$ is the integrand. When $s > 1$, the middle term approaches 0 as $\delta$ approaches 0,therefore yielding
\begin{equation*}
    \begin{split}
        \oint_P e^{-cx}(-x)^{s} \prod_{i=1}^r \frac{2}{1-e^{-a_i(x+j)}} \frac{dx}{x} &= \left[ \lim_{\delta \rightarrow 0}  \oint_{+ \infty}^{\delta} e^{-cx}(-x)^{s} \prod_{i=1}^r \frac{2}{1-e^{-a_i(x+j)}} \frac{dx}{x} \right.\ \\ & +  \left. \oint_{\delta}^{\infty} e^{-cx}(-x)^{s} \prod_{i=1}^r \frac{2}{1-e^{-a_i(x+j)}} \frac{dx}{x} \right] \\
        &= ({e^{i\pi s} - e^{-i\pi s}}) \int_{0}^{\infty} e^{-cx}(x)^{s-1} \prod_{i=1}^r \frac{2}{1-e^{-a_i(x+j)}} dx \\
        &= 2i\sin(\pi s) e^{jc} \Gamma(s)\zeta_E(s,c,j;A_r),
    \end{split}    
\end{equation*}
\begin{equation*}
    \Rightarrow \zeta_E(s,c,j;A_r) = \frac{\Gamma(1-s)}{2\pi i (e^{jc})} \oint_P e^{-cx}(-x)^{s} \prod_{i=1}^r \frac{2}{1-e^{-a_i(x+j)}} \frac{dx}{x} .
\end{equation*}

\end{proof}
    
\begin{lemma}\label{lemma2}
    Assume the same parameters as in lemma \ref{lemma1}, and suppose  $s = -k$ is a negative integer. Then
\begin{equation}
     \zeta_E(-k,c,j;A_r) = \frac{1}{e^{jc}} E_k(c,j;A_r).
\end{equation}
\end{lemma}
\begin{proof}
    When the integral component of the zeta function is partitioned as in lemma \ref{lemma1}, then the left and right integral on the RHS approach 0 as $\delta \rightarrow \infty$ , so by lemma \ref{lemma1}
\begin{equation*}
    \begin{split}
        \zeta_E(-k,c,j;A_r) &= \frac{\Gamma(1+k)}{2\pi i (e^{jc})} \oint_{|x| = \delta} e^{-cx}(-x)^{-k} \prod_{i=1}^r \frac{2}{1+e^{-a_i(x+j)}} \frac{dx}{x} \\
        &= \frac{k!}{2\pi i (e^{jc})} \oint_{|x| = \delta} (-x)^{-k} \sum_{n = 0}^{\infty}E_m(c,j;A_r)  \frac{(-x)^n}{n!}\frac{dx}{x} \\
        & = \frac{k!}{2\pi i (e^{jc})} \sum_{n=0}^{\infty} \frac{(-1)^{n+k}E_m(c,j;A_r)}{n!} \oint_{|x| = \delta} x^{n-k} \frac{dx}{x}, \\
        & = \frac{1}{e^{jc}} E_k(c,j;A_r),
    \end{split}
\end{equation*}
where the last step is implied by the fact that $
   \frac{1}{2\pi i} \oint_{|x| = \delta} x^{n-k} \frac{dx}{x} =1$ when $n=k$ and $0$ otherwise.
\end{proof}
\begin{proof}[Proof of Theorem 1]
    The sum $ \sum_{M_r = 0_r}^{\infty}e^{j(A_r \cdot M_r})(A_r \cdot M_r+x)^{-s}$ is absolutely convergent, so by the Principle of Inclusion Exclusion,
    \begin{equation}\label{eqPIE}
    \begin{split}
        \sum_{M_r = 0_r}^{N_r}\frac{e^{j(A_r \cdot M_r})}{(A_r 
\cdot M_r+x)^s} &= \sum_{S \subseteq \{1,2, \cdots r\}} (-1)^{|S|} \sum_{\substack{0 \leq m_j < \infty \text{ for $j\notin S$}\\ n_j+1 \leq m_j < \infty \text{ for $j\in S$}}}\frac{e^{j(A_r \cdot M_r})}{(A_r 
\cdot M_r+x)^s},  \\
    &= \frac{1}{2^r}  \sum_{S \subseteq \{1,2, \cdots r\}} (-1)^{|S|} \sum_{M_r=0}^{\infty}\frac{e^{j(A_r \cdot M_r + A_s \cdot (N_S+1_{|S|})})}{(A_r 
\cdot M_r+A_S\cdot(N_S+1_{|S|})+x)^s}, \\
        &= \frac{1}{2^r} \sum_{S \subseteq \{1,2, \cdots r\}} (-1)^{|S|}e^{j(A_s \cdot (N_s+1_{|S|}))} \zeta_{E_r}(s,A_s \cdot (N_S+1_{|S|})+x,j;A_r),
    \end{split}   
    \end{equation}
and for a negative integer $s = -m$, analytic continuation implies 
\begin{equation*}
\begin{split}
     \sum_{M_r = 0_r}^{N_r}(A_r 
\cdot M_r+x)^m e^{j(A_r \cdot M_r)} &= \frac{1}{2^r} \sum_{S \subseteq \{1,2, \cdots r\}} (-1)^{|S|} e^{j(A_s \cdot (N_s+1_{|S|}))} \zeta_{E}(-m,A_S \cdot (N_S+1_{|S|})+x,j;A_r), \\
&= \frac{1}{2^r} \sum_{S \subseteq \{1,2, \cdots r\}} (-1)^{|S|} e^{j( A_s \cdot (N_s+1_{|S|})-c)}E_m(A_S \cdot (N_S+1_{|S|})+x,j; A_r),
\end{split}
\end{equation*}
where the last line follows from lemma \ref{lemma2}.
\end{proof}

\section{Proof of Theorem 2}

Note that the periodic function $\tilde{C_{n,k}}(x;a)$ is continuous over the intervals $(\frac{j-1}{k},\frac{j}{k})$ for $j = 1, 2 \cdots , k$, and at $x = \frac{j}{k}$ a jump occurs. In fact, for $x \in  (\frac{j-1}{k},\frac{j}{k})$, $\tilde{C}_{1,k}(x;a)$ is defined as 
\begin{equation*}
    \begin{split}
         \tilde{C}_{1,k}(x;a) &= \sum_{p=0}^{j-1}  e^{2i \pi a p/k}(x-\frac{1}{2} - \frac{p}{k})  + \sum_{q=j}^{k-1}  e^{2i \pi a q/k} (x+\frac{1}{2} - \frac{q}{k}) \\
         &= \sum_{p=1}^{j-1}  e^{2i\pi a p/k}(-\frac{1}{2} - \frac{p}{k}) +\sum_{q=j}^{k-1} e^{2i \pi a q/k}(\frac{1}{2} - \frac{q}{k}).
    \end{split}
\end{equation*}

With this in mind, let $f: \mathbb{R} \rightarrow \mathbb{C} $ be a continuous function and consider the integral
\begin{equation*}
    \begin{split}
        \int_0^1 \tilde{C}_{1,k}(x;a) f'(x) dx &= \sum_{j=1}^k \int_{\frac{j-1}{k}}^{\frac{j}{k}} \tilde{C}_{1,k;a}(x) f'(x)dx \\
        &= \sum_{j=1}^k (f(\frac{j}{k}) - f(\frac{j-1}{k})) \left(  \sum_{p=0}^{j-1}   e^{2i\pi ap/k}(-\frac{1}{2} - \frac{p}{k})  +\sum_{q=j}^{k-1} e^{2i \pi a q/k}(\frac{1}{2} - \frac{q}{k}) \right) \\
        &=  \sum_{j=1}^{k-1} e^{2i\pi a j/k} f(j/k) + f(0) \left(1+ \sum_{q = 0}^{k-1} \frac{q}{k}e^{2i \pi a q/k} \right) - f(1)  \sum_{q=0}^{k-1}\frac{q}{k} e^{2 i\pi a q/k},
    \end{split}
\end{equation*}
which implies that 
\begin{equation*}
    \sum_{j=1}^{k} e^{2i\pi a j/k} f(j/k) = (f(1)-f(0))  \left(1+ \sum_{q = 0}^{k-1} \frac{q}{k}e^{2i \pi a q/k} \right) +  \int_0^1 \tilde{C}_{1,k}(x;a) f'(x) dx,
\end{equation*}
so shifting the equation over for $m$ and $n$ and summing over them telescopes to yields
\begin{equation}
    \sum_{r = m}^n  \sum_{j=1}^{k} e^{2i\pi a j/k} f(r + j/k) = (f(n)-f(m)) \left(1+ \sum_{q = 0}^{k-1} \frac{q}{k}e^{2i \pi aq/k} \right) + \int_m^n \tilde{C}_{1,k}(x;a) f'(x) dx.
\end{equation}
 Note that 
\begin{equation*}
    \begin{split}
        C_{1,k}(a) &= -\frac{1}{2} + \sum_{q=1}^{k-1} e^{2i\pi aq/k}(\frac{1}{2}- \frac{q}{k}) \\
        &= -(1 + \sum_{q = 0}^{k-1} \frac{q}{k}e^{2i \pi q/k}),
     \end{split}
\end{equation*}
so after repeatedly applying integration by parts on (1) and substituting the above yields
\begin{equation*}
    \begin{split}
          \sum_{r = m}^n  \sum_{j=1}^{k} e^{2i\pi aj/k} f(r + j/k) &= (f(n)-f(m)) \left( -C_{1,k}(a)  \right) + \sum_{j=2}^q  \frac{C_{j,k}(a)(-1)^j}{j!} \left[f^{(j-1)}(n) - f^{(j-1)}(m) \right] + R \\
          &= \sum_{j=1}^k  \frac{C_{j,k}(a)(-1)^j}{j!} \left[f^{(j-1)}(n) - f^{(j-1)}(m) \right] + R \\
    \end{split}
\end{equation*}
where $R = \frac{(-1)^{q+1}}{k!} \int_m^n \tilde{C}_k(x) f^{(q)}(x) dx$. Note that we can set the indexing to begin a $j=0$ as $C_0,k;a = 0$, which is useful for clarifying some indexing.

\section{Proof of Theorem 3 and 4}
Here, we shall prove theorem 4 first. To do so, we shall make extensive usage of theorem \ref{theorem2} and follow a similar proof as in \cite{LANPHIER2022125716}. A direct calculation gives that $\frac{\partial^{(j)}}{\partial t^{(j)}}C_{s,m,k}^*(at+x;A_r) = (-1)^j (-s)_j a^j C_{s-j,m,k}^*(at+x;A_r)$. To apply theorem \ref{theorem2} an iterative manner, it is convenient to define the operators
\begin{equation*}
    \phi_{1,a}(C_{-s,m,k}^*(an+x; A)) = \frac{-1}{-s+1} \sum_{j = 0}^{q} (-ak)^{j-1} \binom{-s+1}{j} C_{j,k}(A) C_{-s-j+1,m,k}^* (x;A_r)
\end{equation*}
\begin{equation*}
    \phi_{2,a}(C_{-s,m,k}^*(an+x; A)) = \int_m^n \frac{(-1)^{q+1} \tilde{C}_{q,k}(a t +x; a)}{q!} C_{-s,m,k}^{*(q)}(at+x; A) dt,
\end{equation*}
and extend these definitions linearly so that $\phi_{j,a}$ acts on any multiple of $C_{-s,m,k}^*(an+x; A)$ that is independent of $n$ and $x$. The integral above is absolutely convergent, and since our summations are finite, we have
\begin{equation*}
    \phi_{j,a} \left( \int_0^{\infty} c(t) C_{-s,m,k}^{*(q)}(at+x;A) dt \right) = \int_0^{\infty} c(t) \phi_{j,a} (C_{-s,m,k}^{*(q)}(at+x;A)) dt,
\end{equation*}
where $c(t) = \frac{(-1)^{q+1}}{q!} \tilde{C}_{q,k}(at+x;a)$. From \cite{Lehmer}, $|\tilde{B}_q(t)| \leq 1$, so $|c(t)| \leq \frac{k}{q!}$. Applying theorem \ref{theorem2},
\begin{align*}
    \sum_{n=1}^{\infty} C_{-s,m,k}^*(an+x; A) * e^{2i \pi a n/k} =& \frac{-1}{-s+1} \sum_{j = 0}^{q} (-ak)^{j-1} \binom{-s+1}{j} C_{j,k}(A) C_{-s-j+1,m,k}^* (x;A)\\ & + \int_m^n \tilde{C}_k(x) C_{-s,m,k}^{*(q)}(at+x; A) dt \\
    &= \phi_{1,a}(C_{-s,m,k}^*(an+x; A)) + \phi_{2,a}(C_{-s,m,k}^*(an+x; A)) \\    
\end{align*}
Applying the above iteratively for $l \geq 1$,
\begin{equation*} \label{eq:1}
    \sum_{N_l = 1_l}^{\infty} C_{-s,1,k}^*(A_l \cdot N_l+x; a_l) *  e^{2i \pi A_l \cdot N_l/k}  =  \sum_{\substack{i_j \in \{1,2 \} \\ 1 \leq j \leq l}} (\phi_{i_1,a_1} \circ \cdots \circ \phi_{i_l,a_l})(C_{-s,1,k}(A_l \cdot N_l+x; a_l)).
\end{equation*}
We first consider the term where $i_j = 1$ for all $i_j$. We have 
\begin{equation*}  
    \begin{split}
         \phi_{1,a_l}(C^*_{-s,1,k}(A_l \cdot N_l+x; a_l)) &= \frac{-1}{-s+1} \sum_{j=0}^q (-a_lk)^{j-1} \binom{-s+1}{j}C_j(a_l) C_{-s-j+1,1,k}^*(A_{l-1} \cdot N_{l-1} +x; a_l) \\
         &= C^*_{1,k}(a_l) \sum_{j=0}^q(-k)^j \binom{-s}{j}C_{j,k}^*(a_l)(A_{l-1} \cdot N_{l-1} +x)^{-s-j} \\
         &=  C^*_{1,k}(a_l) C_{-s,q,k}(A_{l-1} \cdot N_{l-1} +x; a_l)
    \end{split}
\end{equation*}
Letting $\phi_{1,a_{l-j}}$ act on $\frac{1}{k^{j-1} (s+1) \cdots (s+j-1)} C_{-s+j-1,q,k}(A_{l-j} \cdot N_{l-j} + x; (a_{l-{(j-1)}}, \cdots, a_l))$, we have
\begin{align*}
   & \frac{1}{k^{j-1} (s+1) \cdots (s+j-1)} \frac{-1}{s+j} \\
   & \times \sum_{p =0}^q \binom{-s+j}{p} (-a_{l-j}k)^{p-1} C_{p,k}(a_{l-j})  C_{-s+j - p,q,k}(A_{l-j} \cdot N_{l-j} + x; (a_{l-{(j-1)}}, \cdots, a_l)) \\
   &= \frac{1}{k^{j} (s+1) \cdots (s+j)}  \sum_{p =0}^q \binom{-s+j}{p} (-k)^{p} C_{p,k}^*(a_{l-j})  \\
   & \times \left(\sum_{i=0}^q C_{i,k}^*(a_{l-{j-1}}, \cdots, a_l) (-k)^i \binom{-s+j-p}{i} (A_{l-j} \cdot N_{l-j}+x)^{-s+j-p-i})\right) \\
   &= \frac{1}{k^{j} (s+1) \cdots (s+j)} \sum_{u=0}^q \binom{-s+j}{u} (-k)^u \\
   & \left( \sum_{p+i = u} \binom{u}{k} C_{p,k}^*(a_{l-j})  C_{i,k}^*(a_{l-(j-1)}, \cdots, a_l)  \right)  (A_{l-j} \cdot N_{l-j}+x)^{-s+j-p-i}     \\
   &= \frac{1}{k^{j} (s+1) \cdots (s+j)} \sum_{u=0}^q \binom{-s+j}{u} (-k)^u
   C_{i,k}^*(a_{l-j}, \cdots, a_l)  (A_{l-j} \cdot N_{l-j}+x)^{-s+j-u} \\
   &=  \frac{1}{k^{j} (s+1) \cdots (s+j)} C_{-s+j,q,k}^* (A_{l-j} \cdot N_{l-j}+x;(a_{l-j}, \cdots, a_l)),
\end{align*}
which implies 
\begin{equation*}
    (\phi_{1,a_1} \circ \cdots \circ \phi_{1,a_l}) (C_{-s,1,k}(A_l \cdot N_l +x;a_l)) = C^*_{1,k}(a_l) * \frac{1}{k^{l-1} (s+1) \cdots (s+l-1)} C_{-s+l-1,q,k}^* (x; A_l).
\end{equation*}
Now consider the terms in equation $\ref{eq:1}$ containing some $\phi_{2,a_j}$. For such a term, let $S$ be the subset of $\{i_1, \cdots, i_l \}$ where $i_j = 2$ if and only if $i_j \in S$. From the previous computation, the action of all of the $\phi_{1,a_j}'s$ on $C_{-s,1,k} (A_l \cdot N_l + x; a_l)$ gives $C_{-s+l-|S|-1}(A_s \cdot N_s + x; A_{S^*})$, where $S^*$ is the subset of $\{i_1, \cdots i_l \}$ where $i_j = 1$ if and only if $i_j \in S^*$. By definition of $\phi_{2,a_j}$, we have
\begin{align*}
    & (\phi_{i_1,a_1} \circ \cdots \circ \phi_{i_l,a_l})(C_{-s,1,k}(A_l \cdot N_l+x; a_l)) \\
    &= M \underbrace{\int_0^{\infty} \cdots \int_0^{\infty}}_{|S| \text{ times}} c(\bold{t}_S) \frac{\partial^{(q+l-|S|)}}{\partial \bold{t}_S^{(q+l-|S|)}} C_{-s+l-|S|-1, q, k}(A_s \cdot \bold{t}_S+x ; A_{S^*}) d\bold{t}_S \\
\end{align*}
where $M$ is a constant independent of $x$, $\bold{t}_S = (t_{i_1 '}, \cdots, t_{i_{|S|}'})$, and $ \frac{\partial^{(q+l-|S|)}}{\partial \bold{t}_S^{(q+l-|S|)}} =  \frac{\partial^{(q+l-|S|)}}{\partial t_{i_1}'^{(q+l-|S|)}} \cdots  \frac{\partial^{(q+l-|S|)}}{\partial t_{i_{|S|}}'^{(q+l-|S|)}}$. Because $|\tilde{B}_n|$ is bounded,  $|\tilde{C}_{n,k}|$ is bounded and thus $|c(\bold{t}_S)|$ is bounded, so it is fairly straightforward to see that the expression above is bounded above by 
\begin{equation*}
    M'  \underbrace{\int_0^{\infty} \cdots \int_0^{\infty}}_{|S| \text{ times}} (A_S \cdot \bold{t}_S+x)^{-s+l - |S|(q+l-|S|)} d\bold{t}_S,
\end{equation*}
where $M'$ is independent of $x$. This expression is $x^{-s+l - |S|(q+l-|S|-1)}$ multiplied by a term independent of $x$. As $-s+l - |S|(q+l-|S|-1) \leq -s -|S|(q+1) \leq -s-q+1$ for $1 \leq |S| \leq r$, we can see that any term in $\ref{eq:1}$ with some $i_j = 2$ lies in $O(x^{Re(-s)-q+1})$. It follows that if we take $l = r$ we have, for $Re(s) > r$,
\begin{equation*}
    \begin{split}
         \zeta_E(s,x,\frac{2\pi i}{k};A_r) &= 2^r \sum_{N_r = 0_r}^{\infty} (A_r \cdot N_r +x)^{-s} * e^{\frac{2i\pi A_r \cdot N_r}{k}} \\
         &= 2^r \sum_{N_r = 1_r}^{\infty} (A_r \cdot N_r +x - A_r \cdot 1_r)^{-s} * e^{\frac{2i\pi A_r \cdot N_r}{k}}  \\
         &= 2^r \frac{1}{C_{1,k}(a_r)} \sum_{N_r = 1_r}^{\infty} C_{-s+1,1,k}^*(A_r \cdot N_r+x - A_r \cdot 1_r; a_r) *e^{\frac{2i\pi A_r \cdot N_r}{k}}  \\
         &= 2^r \frac{ \Gamma(-s+2)}{k^{r-1} \Gamma(-s+r+1)}  C_{-s+r,q,k}^* (x - A_r \cdot 1_r; A_r) + O(x^{Re(-s)-q+2}).
    \end{split}
\end{equation*}
where for $s \in \mathbb{Z}_{> r}$ the ratios of the gamma functions are determined by taking residues. By analytic continuation, Applying equation (\ref{eqPIE}) and the above, we have
\begin{equation*}
    \begin{split}
        \sum_{M_r = 0_r}^{N_r} (A_r \cdot M_r+x)^s * e^{\frac{2\pi i}{k}(A_r \cdot M_r)} &=\frac{1}{2^r} \sum_{S \subseteq \{1,2, \cdots r\}} (-1)^{|S|}e^{\frac{2\pi i}{k}(A_s \cdot (N_s+1_{|S|}))} \zeta_{E_r}(-s,A_s \cdot (N_S+1_{|S|})+x,j;A_r) \\
        &= \frac{ \Gamma(s+2)}{k^{r-1} \Gamma(s+r+1)} \\
        & \times \sum_{\substack{S \subseteq \{1,2, \cdots r\} \\ S \neq \emptyset}} (-1)^{|S|}e^{\frac{2\pi i}{k}(A_s \cdot (N_S+1_{|S|}))} C_{s+r,q,k} (A_S \cdot N_S+x;A_r) \\
        & + \zeta_E(-s, x; A_r) + O(N_{min}^{Re(s)-q+2}),
    \end{split}
\end{equation*}
proving theorem 3.

\printbibliography
\end{document}